\documentclass[a4paper,11pt]{article}

\newcounter{remark}[section]

\usepackage[top=30mm, bottom=30mm, left=45mm, right=35mm]{geometry}

\usepackage{amsmath}
\usepackage{amssymb}
\usepackage{bm}
\usepackage{mathrsfs}
\usepackage[dvips]{graphicx}

\usepackage{color}


\DeclareMathAlphabet{\mathsfsl}{OT1}{cmss}{m}{sl}

    {\makebox{\phantom{}} \\ %
        \textsc{Input:}\begin{itemize}}%
    {\end{itemize}}

    {\textsc{Output:}\begin{itemize}}%
    {\end{itemize}}

    {\textsc{Procedure:}\begin{enumerate}}%
    {\end{enumerate}}

\renewcommand{\phi}{\varphi}

\usepackage{scalerel,stackengine}
\stackMath
\newcommand\reallywidehat[1]{%
\savestack{\tmpbox}{\stretchto{%
  \scaleto{%
    \scalerel*[\widthof{\ensuremath{#1}}]{\kern-.6pt\bigwedge\kern-.6pt}%
    {\rule[-\textheight/2]{1ex}{\textheight}}
  }{\textheight}%
}{0.5ex}}%
\stackon[1pt]{#1}{\tmpbox}%
}
\usepackage{graphicx}
\usepackage[backend=bibtex,natbib=true,maxbibnames=99]{biblatex}
\usepackage{url}
\usepackage{citesort}
\usepackage{rotating}
\usepackage{geometry}
\usepackage{float}
\usepackage{hyperref}
\usepackage{algorithmic}

\usepackage{enumitem,nicefrac}   
\usepackage{graphics}
\usepackage{caption}\usepackage{subcaption}
\usepackage{amsthm}

\usepackage{cleveref}

\bibliography{Literature.bib} 

\newtheorem{theorem}{Theorem}[section]

\newtheorem*{problem*}{Problem}

\newtheorem{proposition}[theorem]{Proposition}

\theoremstyle{definition}

\theoremstyle{remark}
\newtheorem{remark}[theorem]{Remark}
\numberwithin{equation}{section}

\newcommand{\E}{\mathbb{E}}
\newcommand{\R}{\mathbb{R}}

\newcommand{\bs}{\boldsymbol}
\newcommand{\btheta}{\bs\theta}
\newcommand{\bzeta}{\bs\zeta}

\newcommand{\bfSigma}{\bs\Sigma}

\begin{document}

\title{Non-Asymptotic Guarantees For Sampling by \\ Stochastic Gradient Descent}

\author{ Avetik Karagulyan \footnote{ENSAE/CREST, avetik.karagulyan@ensae.fr}}







\maketitle
\begin{abstract}
Sampling from various kinds of distributions is an issue of paramount importance
in statistics since it is often the key ingredient for constructing estimators, test procedures or
confidence intervals. In many situations, the exact sampling from a given distribution is impossible or computationally expensive and, therefore, one needs to resort to approximate sampling
strategies. However, it is only very recently that a mathematical 
theory providing non-asymptotic guarantees for approximate sampling 
problem in the high-dimensional settings started to be developed. 
In this paper we introduce a new mathematical framework that helps to 
analyze the Stochastic Gradient Descent as a method of sampling, closely related to Langevin Monte-Carlo.
\end{abstract}

\section{Introduction}
Let us first introduce the mathematical setting of Langevin sampling. The general problem  is to sample from the log-concave distribution with density $\pi(\theta) = c \exp(-f(\theta)) $, where $f:\R^p \rightarrow \R$ satisfies the following two conditions:
\begin{align}\label{str_conv}
&\text{Strong convexity}:  f(\theta_2)\geq f(\theta_1) + \nabla 
f(\theta_1)^T(\theta_2 - \theta_1) + \frac{m}{2} \|\theta_1 - \theta_2\|^2_2; \\ \label{lipschitz}
&\text{Smoothness}: \|\nabla f(\theta_1)-\nabla f(\theta_2)\|_2 \leq M\|\theta_1 - \theta_2\|_2,
\end{align}
for all $p$-dimensional real vectors $\theta_1$ and $\theta_2$.
The parameters $m$ and $M$ are positive numbers and $\|\cdot\|_2$ is 
the Euclidean norm on $\R^p$.
The problem of sampling from $\pi$ is closely related to the problem of finding
the minimum of the function $f:\R^p \rightarrow \R$. Indeed, suppose we 
manage to sample from the distribution $\pi_\beta(\theta) = c_\beta \exp(-\beta f(\theta))$, where $\beta$ is a large positive number. Then $\pi_\beta$ 
will mainly be concentrated around the unique minimum point of $f$ and 
it will have some kind of a spike form. Thus, a sample from $\pi_\beta$ 
is a high probability approximation of the minimum point. Therefore  
considering $f$ to be convex will facilitate our task for 
characterizing the convergence of the considered sampling method. For more details see 
\citep{dalalyan2017theoretical} and \citep{gelfand1991recursive}.

Langevin Monte-Carlo algorithm is one of the methods for the approximate sampling from
 the target distribution $\pi$. The idea comes from the 
following Stochastic Differential Equation  (SDE), named Langevin 
diffusion:
\begin{equation}\label{SDE}
dX (t) = -\nabla f (X(t))dt + \sqrt{2}dW(t).
\end{equation}
Here $W$ is the standard Wiener process or Brownian motion in 
$\R^p$. The solution of \eqref{SDE} is a Markov process having $\pi$ 
as invariant distribution \citep{BAKRY2008727}. In order to use this fact for our 
goal, we will use Euler-Maruyama discretization of \eqref{SDE}, 
which can be found in \citep{roberts1996exponential}. It goes as follows:
\begin{equation}\label{LMC}
\theta_{k+1} = \theta_k - h_{k+1}\nabla f(\theta_k) + \sqrt{2h_{k+1}} \xi_{k+1},
\end{equation}
where $\xi_1,\xi_2,\ldots,\xi_k,\ldots$ follow Gaussian distribution $\mathcal{N}(0,I_p)$ and are independent from each 
other and $\theta_0$. The latter is the starting point for the algorithm and it can be 
random as well. In particular when the step-sizes $h_k$ are constantly equal to $h$ and 
$h$ is small, then for large enough $k$'s the distance (Wasserstein, 
Total Variation) between the distribution of $\theta_k$ and $\pi$ is 
small. This algorithm is called {Gradient Langevin  Dynamics} (GLD) 
or Langevin Monte-Carlo (LMC) and it is actively studied nowadays 
(\cite{brosse2017sampling}; \cite{chatterji2018theory};
\cite{dalalyan2017theoretical}; \cite{dalalyan2017user};
\cite{durmus2016high}; \cite{durmus2017nonasymptotic}).

In this paper, however we are not going to study the convergence of 
LMC algorithms. Instead we will review Stochastic Gradient Descent 
as a sampling method and represent it as a sampling algorithm. Let us recall SGD for the case of optimization. Often in Machine Learning problems we need to minimize the empirical risk. The latter is usually a sum-decomposable function $f:\R^p \rightarrow \R$:
\begin{equation}
f(x) = \sum_{i=1}^{n}g_i(x),
\end{equation}
where $n$ is the sample size and $g_i:\R^p \rightarrow \R$, for 
every $i=1,\ldots,n$.  The classical algorithm to solve a 
minimization problem, when mild assumptions are satisfied, is the 
Gradient Descent. Unfortunately when the sample size is large then 
every step of Gradient Descent is becoming computationally 
expensive. That is why Stochastic Gradient Descent is introduced. 
The main idea of SGD is to replace the full gradient in GD with its 
unbiased estimate. There are various ways to do it, but the most 
common one is the so called Batch Gradient Descent. In the latter 
case, one just samples a mini-batch $B$ (a subset of $\{1,2,\ldots,n\}$) and replaces the gradient by $c_B\sum_{i\in B}\nabla g_i$, where 
$c_B$ is a constant depending on $|B|$. Thus the update rule becomes
$\theta_{k+1} = \theta_{k} - c_B\sum_{i\in B}\nabla g_i $.  For more details  see \citep{bottou2018optimization}. 

The problem of our interest however is not directly related to 
optimization, but to sampling
 We will show that in the case of a smooth and strongly convex 
 potential function $f$ SGD yields a convergence of 
order ${\tilde{O}({\kappa^2p}/{\epsilon^2})}$
\footnote{$\tilde{O}$ is the  big-$O$ notation, ignoring logarithmic factors.}
in Wasserstein error. If in addition to these conditions we also 
have second-order smoothness, then the rate improves to 
${
\tilde{O}(\kappa^2 p/\epsilon^2 \bigwedge \kappa \sqrt{np}/\epsilon)}$.

This article is organized as follows: In the next section, we give 
some remarks about the past and ongoing research in this area. 
\Cref{sec:3} gives some notions 
about the prior work in Langevin sampling. Next, in \Cref{sec:4} we 
introduce the theoretical setting that we are going to work with.
In the proceeding section we propose a mathematical framework which 
helps to analyze the convergence. The main results 
that provide non-asymptotic upper bounds to convergence rate are 
presented in \Cref{sec:6}.

\section{Prior work}\label{sec:2}

The first and probably the most influential work providing probabilistic analysis of the asymptotic properties of the LMC algorithm is decsribed in \citep{roberts1996exponential}. However, one of the
recommendations made by the authors of that paper is to avoid using Langevin algorithm
as it is defined in \eqref{LMC} or to use it very cautiously, since the 
ergodicity of the corresponding Markov chain  $\theta_k$ is very sensitive 
to the choice of the parameter $h$. Even in the cases where the 
Langevin diffusion is geometrically ergodic, the inappropriate choice of $h$ 
may result the
transience of the Markov chain.  These findings have  strongly 
influenced the subsequent studies since all the ensuing research 
focused essentially on the 
Metropolis adjusted version of the LMC, known as Metropolis adjusted 
Langevin algorithm (MALA) and its numerous modifications (\cite{Jarner2000}; 
\cite{Pillai2012}; \cite{RobertsStramer02}; \cite{RobertsRosenthal98}; 
\cite{StramerTweedie99-1}). In contrast to this, it is shown that under the strong 
convexity assumption imposed on $f$ coupled with the Lipschitz 
continuity of the gradient of $f$, one can ensure the non-transience 
of the Markov chain  $\theta_k$ by a suitable choice of $h_k$. Later 
by \cite{dalalyan2017theoretical} and \cite{durmus2016high} it was 
shown that the convergence rate in TV distance is $\tilde{O}(p/\epsilon^2)$ 
for any initial vector $\theta_0$.

Another problem of interest is the convergence in Wasserstein 
distance. In the next section the reader can find our reasoning to choose Wasserstein distance instead of TV. The convergence of LMC with this 
error was recently studied by \citep{dalalyan2017user} and 
\citep{durmus2016high} and a rate of $\tilde{O}(p/\epsilon^2)$ was 
achieved. In addition to this, in \citep{dalalyan2017user} it was 
shown, that imposing additional smoothness for function $f$, meaning 
Lipschitz-continuity of its Hessian matrix, implies a better 
convergence rate of $\tilde{O}(\sqrt{p}/\epsilon)$ for LMC.  
It turns out that in the case of sum-decomposable potential 
function, 
a modified version of LMC achives a better convergence rate. Some of 
these algorithms have their roots in optimization, like SAGA 
\citep{chatterji2018theory}, which was originally proposed in a 
paper by Defazio et al. \citep{DefazioBL14}  for the problem of 
optimization.

\section{Preliminaries}\label{sec:3}

The convergence in terms of Wasserstein error was studied by many 
authors. \cite{durmus2016high} proved the rate  $O(p/\epsilon^2)$  
for any deterministic starting point $\theta_0$. The same 
convergence with improved coefficients was later shown in 
\citep{dalalyan2017user}.
In this section we 
will formulate two theorems from \citep{dalalyan2017user}, which will be used later on. 
Before we state the theorems, let us define $W_2$ 
Wasserstein distance. For two probability measures $\mu$ and $\nu$ defined on $(\R^p, 
\mathcal{B}(\R^p))$, $W_2$ distance is defined by 
\begin{equation}
W_2(\mu,\nu) = \bigg\{\inf\limits_{\eta \in \Gamma(\mu,\nu)} \int_{\R^p\times\R^p} \|\theta - \theta'\|_2^2 d\eta(\theta,\theta')\bigg\}^{\frac{1}{2}},
\end{equation}
where the infimum is taken with respect to all joint distributions 
$\eta$ 
having $\mu$ and $\nu$ as marginal distributions. Let us compare 
this distance to total variation distance. If we have small 
Wasserstein for some $\mu$ and $\nu$, then it implies  that their 
first order moments are also close. This property does not hold 
for the total variation distance. As an example one can check 
that $\|\delta_\theta - \delta_\theta'\|_{TV} = \textbf{1}
_{\theta \neq \theta'}$, whereas $W_2(\delta_\theta,\delta_
\theta') = \|\theta - \theta'\|_2$ is a smooth function 
increasing function of Euclidean distance between $\theta$ and $\theta'$. 

Let us now present a non-asymptotic convergence bound for Wasserstein error, when the constant step-size LMC . 
\begin{theorem}[Theorem 1 from \citep{dalalyan2017user}]\label{user-friendly}
Assume that $h \in (0,2/M)$. Let $f$ satisfy conditions  
\eqref{str_conv} and \eqref{lipschitz} , thus the following claims hold:
\begin{align*}
& \text{if}\hspace{0.2cm} h\leq \frac{2}{m+M},\hspace{0.2cm} \text{then} \hspace{0.2cm}W_2(\nu_K,\pi) \leq (1-mh)^K W_2(\nu_0,\pi) + \frac{1.65M}{m}(hp)^{\frac{1}{2}}; \\
& \text{if}\hspace{0.2cm} h\geq \frac{2}{m+M},\hspace{0.2cm} \text{then} \hspace{0.2cm}W_2(\nu_K,\pi) \leq (Mh-1)^K W_2(\nu_0,\pi) + \frac{1.65Mh}{2-Mh}(hp)^{\frac{1}{2}}.
\end{align*}
\end{theorem}


In practice, a relevant approach to get an accuracy of at most $\epsilon$ is to minimize the
upper bound provided by \Cref{user-friendly} with respect to $h$, for a fixed K. Then, one can choose
the smallest $K$ for which the obtained upper bound is smaller than $\epsilon$. One useful observation
is that the second upper bound is an increasing function of $h$. Its 
minimum is always attained at $h = 2/(m + M)$, which means that one can always look for a 
step-size in the interval $(0, 2/(m+M)]$ by minimizing the first upper bound. This can be done using standard methods of optimization.
\begin{remark} These two upper bounds contain $W_2(\nu_0,\pi)$, computation of which can be involving. In order to avoid it, we will bound it from above. If $f\ge 0$, we can replace it by 
$\sqrt{p/m} + \sqrt{2f(\theta_0)/m}$. Indeed,
\begin{align*}
W_2(\nu_0,\pi) &\leq \sqrt{\frac{p}{m}} +  \|\theta_0-\theta\|_2\\
&\leq \sqrt{\frac{p}{m}} +  \sqrt{\frac{2}{m}\big(f(\theta_0) - f(\theta_*) \big)}\\
&\leq  \sqrt{\frac{p}{m}} +  \sqrt{\frac{2f(\theta_0)}{m}}.
\end{align*} 
\end{remark}
The first inequality is a corollary from Proposition 1 of  \citep{durmus2016high}. Combining \Cref{user-friendly} with its remarks we obtain the following. Suppose that we choose $h$ and $K$ so that
\begin{equation}\label{eps-cond} 
h\leq \min\Big(\frac{2}{m+M},\frac{m^2\epsilon^2}{11M^2p}\Big) \hspace{0.2cm} \text{and}  
\hspace{0.2cm} hK \geq \frac{1}{m}\log\left(Q(p,\epsilon)\right),
\end{equation} 
where $$Q(p,\epsilon) = \frac{2 f(\theta_0) + mp}{0.5m\epsilon}$$ is a real-valued rational function.  Then each of the 
components from the right-hand side of the theorem will be less than  $0.5\epsilon$, thus $W_2(\nu_K,\pi) \leq \epsilon.$ 
\subsection{Non-asymptotic guarantee with second-order smoothness}
Below we present a theorem that quantifies the non-asymptotic 
behavior of LMC, when the potential function has a
 Lipschitz-continuous Hessian. That is, for every $x,y \in \R^p$ we have
\begin{equation}\label{hessian-lipschitz}
\|\nabla^2 f(x) - \nabla^2 f(y)\|\leq L\|x-y\|_2,
\end{equation}
where $\|\cdot \|$ is the operator norm of matrices. 
\begin{theorem}[Theorem 4 from \citep{dalalyan2017user}]\label{user_sec_ord}
Let $\nu_K$ be the distribution of $K$-th iterate of the LMC 
algorithm iterations. Assume that the function $f:\R^p \rightarrow 
\R$ satisfies \eqref{str_conv}, \eqref{lipschitz} and it is also 
$L$-Hessian-Lipschitz. Then for 
every $h < \nicefrac{2}{(m+M)}$,
\begin{equation*}
W_2(\nu_K,\pi) \leq (1-mh)^K W_2(\nu_0,\pi) + \frac{Lhp}{2m} + \frac{11M^{\frac{3}{2}}h\sqrt{p}}{5m}.
\end{equation*}
\end{theorem} 
\begin{remark}
In order for the improvement of the rate to be visible, let us take a closer look to the order of step-size $h$ and dimension $p$. Here we have $O(hp)$ meanwhile \Cref{user-friendly} gives only $O(\sqrt{hp})$, which is worse as $h$ is considered to be small.
\end{remark}
\begin{remark}
Doing analogous analysis as we did for the previous theorem, one can deduce that the convergence rate is $\tilde{O}(\sqrt{p}/\epsilon)$.
\end{remark}

\section{Proposed framework to analyze SGD}\label{sec:4}

In  the following sections we will discuss a special case for potential function $f$, in particular when $f$ is a sum-decomposable function. That is:
\begin{equation}
f(\btheta) = \sum_{i=1}^n g(\btheta,Z_i),
\end{equation}
where $n$ is a very large positive integer, $g:\mathbb R^p\times \mathcal Z\to\mathbb R$
is a given smooth function and $Z_1,\ldots,Z_n$ are iid random variables with values
in some probability space $\mathcal Z$. To ease notation, we write 
 $g_i(\btheta) = g(\btheta,Z_i)$. We assume here that the 
functions $g_i$ are strongly convex with a coefficient $m_g$ and its 
gradient is $M_g$ Lipschitz-continuous.
 Therefore $f$ is a convex and gradient-Lipschitz function as well, 
 with coefficients $nm_g$ and $nM_g$.  So we have
\begin{equation}
\nabla f(\btheta) = \sum_{i=1}^n \nabla g_i(\btheta).
\end{equation}
In order to avoid the computation of $n$ gradients $\nabla g_i$ at each
iteration of the LMC, we will use the 
classic Stochastic Gradient Descent algorithm in order to sample 
approximately. Let us first recall the algorithm. At each iteration 
$k$ of the algorithm, we choose a subset $B_k$ independent of all 
the past randomness and update $\theta_{k+1}$ by
\begin{align}
\theta_{k+1} = \theta_k - \frac{hn}{b}\sum_{i\in B_k} \nabla g_i(\btheta_k).
\end{align}
The latter can be rewritten as 
\begin{equation}
\theta_{k+1} = \theta_k - h\nabla f(\theta_k) + h\bzeta_k, 
\end{equation}
where the noise vectors $\bzeta_k$ are of the form
\begin{align}
\bzeta_k = n\bigg\{\frac1b \sum_{i\in B_k} \nabla g_i(\btheta_k) - 
\frac1n \sum_{i=1}^n \nabla g_i(\btheta_k) \bigg\}.
\end{align}
If $b$ is large, the distribution of $\bzeta_k$ (conditionally to $\btheta_k$) 
is approximately Gaussian $\mathcal N_p(\mathbf 0, \bfSigma_k)$
where the covariance matrix $\bfSigma_k$ is given by
\begin{align}
\bfSigma_k = \frac1n\sum_{i=1}^n \nabla g_i(\btheta_k)\nabla g_i(\btheta_k)^\top
-\bigg\{\frac1n\sum_{i=1}^n \nabla g_i(\btheta_k)\bigg\}
\bigg\{\frac1n\sum_{i=1}^n \nabla g_i(\btheta_k)\bigg\}^\top.
\end{align}

Below we study a particular case of SGD when the noise vector $
\zeta_k$ is a normal random vector with a covariance proportional to 
identity matrix. We will assume, that $\bfSigma_k = \sigma^2 I_p$, 
where $\sigma^2 = n(n-b)/b$. The choice of $\sigma^2$
is intuitive. For details see the Appendix. Let us 
formulate the framework we are going to work with. 

\textbf{Assumptions:} 
Suppose $g_i:\R^p \rightarrow \R^p$ for $i=1,\ldots,n$ and $f = \sum_i 
g_i$. We will assume that the functions  $g_1,g_2,\ldots,g_n$
satisfy the assumptions \eqref{str_conv} and \eqref{lipschitz} with 
coefficients $m_g$ and $M_g$, respectively. 

\textbf{Iterative method:} \begin{equation}
\theta_{k+1} = \theta_k - h \nabla f(\theta_k) + h\zeta_k,
\end{equation}
where
\begin{equation}
\zeta_k \sim \mathcal{N}\left(0,\frac{n(n-b)}{b}I_p\right),
\end{equation}
for every $k=1,2,\ldots,n$.

\textbf{Problem:} Find a solution to this optimization problem
$$\begin{cases}%
\textbf{Minimize } Kb ; \\
\textbf{Subject to } \min\limits_hW_2(\nu_{K,h,b},\pi) \leq 
\epsilon,
\end{cases}$$
where $\nu_{K,h,b}$ is the distribution of the $K$-th iterate of the SGD with step-size $h$ and batch-size $b$.
In other words, what is the minimum amount of overall gradient evaluations in order to have an error of $\epsilon$.
\section{Main results}\label{sec:5}
In this section we present two theorems that solve the problem stated above in two slightly different cases. For the rest of the paper we define the condition number $M_g/m_g$ by $\kappa$.
\begin{theorem}\label{theorem-1} 
Suppose that the following conditions are satisfied:
\begin{equation}
h {= } \frac{\epsilon^2}{4\kappa^2p}, \hspace{0.2cm} 
b = \frac{hn^2}{2+hn},\hspace{0.2cm}  n\geq 9 \hspace{0.2cm} \text{ and }\hspace{0.2cm} \,
 \frac{3\kappa\sqrt{p}}{n} \leq \epsilon \leq \frac{2\kappa\sqrt{p}}{\sqrt{nM_g}}.
\end{equation}
If
\begin{equation}\label{cond-T}
Kb \geq \frac{4p\kappa^2n\log(Q'(p,\epsilon))}{m_g(8p\kappa^2 + \epsilon^2n)} ,
\end{equation}
where  $Q'$ is a {rational function} given by formula
$$Q'(p,\epsilon) = \frac{2 f(\theta_0) + m_g p}{0.1m_g\epsilon},$$ then 
\begin{equation}
W_2(\nu_{K,h,b},\pi)\leq \epsilon.
\end{equation}
\end{theorem}
Before we bring the proof let us state some remarks regarding this theorem.
\begin{remark}
Since the batch-size $b$ is between $1$ and $n$, $\nicefrac{hn^2}{(2+hn)}$ must also satisfy this condition. In order to verify that, let us substitute $h$ with its value.  Therefore we have
\begin{equation}
b = \frac{n^2\epsilon^2}{8\kappa^2p + n\epsilon^2}\cdot 
\end{equation}
The latter is a monotonically increasing function with respect to 
$\epsilon^2$. Thus taking into account that $n$ is larger than $9$, 
\begin{equation}
b = \frac{n^2}{\frac{8\kappa^2p}{\epsilon^2} + n} \geq \frac{n^2}{\frac{8n^2}{9} + n} \geq 1.
\end{equation}
The inequality $b\leq n$ is obvious.
\end{remark}
\begin{remark}
One can notice that, if $n\rightarrow \infty$, then $Kb$ has an order of $\tilde{O}\left(\frac{4p\kappa^2}{\epsilon^2}\right)$.  
\end{remark}
\begin{proof}
As the function $f$ is a sum of $n$ strongly-convex and 
gradient-Lipschitz functions, then it is also a strongly-convex and 
gradient-Lipschitz function with coefficients $m = nm_g$ and $M = nM_g$, respectively.
First let us express the step-size $h$ in terms of the batch-size $b$.
From the formula of $b$, we obtain
\begin{equation}
h = \frac{2b}{n(n-b)}.
\end{equation}
Thus if we can rewrite the iterative method in the following way:
\begin{align*}
\theta_{k+1} &= \theta_k - h \nabla f(\theta_k) - h\zeta_k \\
&= \theta_k - h \nabla f(\theta_k) +
h\sqrt{\frac{n(n-b)}{b}}\eta_k\\%
&= \theta_k - h \nabla f(\theta_k) + \sqrt{2h}\eta_k,
\end{align*}
where $\eta_1,\eta_2,\ldots$, as usual, are independent standard 
normal $p$-dimensional random vectors.  Therefore we got the classic 
LMC update rule. From the definition of $h$ we have 
\begin{equation}
h {=}  \frac{\epsilon^2}{4\kappa^2p} \leq \frac{1}{nM_g}.
\end{equation} Thus \Cref{user-friendly} yields
\begin{equation}
W_2(\nu_K,\pi)\leq (1-nm_gh)^KQ(p,\epsilon) + 1.65\kappa\sqrt{ph}.
\end{equation}
We will give upper bounds for each component of the right-hand side.
Substituting $h$ with its value in $\kappa\sqrt{ph}$ we obtain, that
\begin{equation}
\kappa\sqrt{ph} = \kappa\sqrt{p\frac{\epsilon^2}{4\kappa^2p}} =\frac{\epsilon}{2}.
\end{equation} Now let us discuss the other component.
As we mentioned in previous sections, if 
\begin{equation}\label{cond-K}
K \geq\frac{\log(Q'(p,\epsilon))}{m_gnh} =\frac{4p\kappa^2}{m_gn\epsilon^2} \cdot \log(Q'(p,\epsilon)),
\end{equation}
then $(1-nm_gh)^KW_2(\nu_{0,h,b},\pi)$ will be less than $0.1\epsilon$. 
In order to complete the proof we just need to multiply this lower bound on $K$ by $b$.  Thus we obtain 
\begin{equation}\label{Kb-bound-nax}
Kb \geq \frac{4p\kappa^2b}{m_gn\epsilon^2} \cdot \log(Q'(p,\epsilon)).
\end{equation}
Using the definition of $h$, we obtain the following formula for $b$
\begin{equation}
b = \frac{n^2\epsilon^2}{8\kappa^2p + n\epsilon^2}.
\end{equation} 
Substituting the latter in \eqref{Kb-bound-nax}, we get the required .
\end{proof}

{\subsection{Convergence of SGD with second-order smoothness}}
In this section we will analyze the convergence of Stochastic 
Gradient Descent in terms of Wasserstein distance when the Hessian 
matrix of the function $f$ is Lipschitz-continuous. 
\begin{theorem}
Suppose that the following conditions are satisfied:
\begin{equation*}
h = \frac{\epsilon}{4\kappa L_g\sqrt{M_gp\max(p,n)}}, \hspace{0.5cm} 
b = \frac{hn^2}{2+hn},
\end{equation*}
\begin{equation*}
\frac{2\sqrt{p\max(p,n)}}{n(n-1)} \leq \frac{\epsilon}{4\kappa L_g\sqrt{M_g}} \leq \frac{\sqrt{p\max(p,n)}}{M_gn}.
\end{equation*}
If 
\begin{equation}\label{cond-Kb}
Kb \geq \frac{4n\kappa L_g\sqrt{M_gp\max(p,n)}}{m_g(8\kappa L_g\sqrt{M_gp\max(p,n)} + n\epsilon )}\cdot\log(Q''\left(p,\epsilon)\right),
\end{equation}
where $$Q''(p,\epsilon) = \frac{2 f(\theta_0) + m_g p}{0.3m_g\epsilon}$$
then 
\begin{equation}
W_2(\nu_{K,h,b},\pi)\leq \epsilon.
\end{equation}
\end{theorem}
\begin{remark}
Again the condition on $\epsilon$ is brought to make the choice of 
parameters possible. In particular, as mentioned before, $b$ is an 
integer between $1$ and $n$.  Doing simple calculations and using 
the aforementioned condition, one can verify that our formula $b$ 
satisfies this criteria.
\end{remark}

\begin{remark}
Let us interpret a little the result of the theorem.  In the case when our sample size $n$ tends to infinity, we have $O\left(\kappa\sqrt{np}\log\left(Q(p,\epsilon)\right)/\epsilon\right)$ complexity.
\end{remark}

\begin{proof} The proof is similar to the one for \Cref{theorem-1}. Using the same reasoning as before $f$ satisfies \eqref{str_conv}, \eqref{lipschitz}, \eqref{hessian-lipschitz} with $m = nm_g$, $M = nM_g$ and $L = nL_g$, respectively.
As in the previous proof we will represent our iterative method as a 
classic Langevin Monte-Carlo update step. 
We have that
\begin{equation}
h = \frac{\epsilon}{4\kappa L_g\sqrt{M_gp\max(p,n)}} \leq \frac{1}{nM_g},
\end{equation}
therefore \Cref{user_sec_ord} can be applied:
\begin{equation}
W_2(\nu_{K,h,b},\pi)\leq  (1-nm_gh)^KW_2(\nu_{0	,h,b},\pi) + \frac{L_ghp}{2m_g} + \frac{11}{5}\kappa h\sqrt{M_gpn}.
\end{equation}
 Let us express $b$ in terms of $\epsilon$, $p$ and $n$:
\begin{equation}
b = \frac{hn^2}{2+hn} = \frac{n^2}{\frac{2}{h} + n} =  \frac{\epsilon n^2}{8\kappa L_g\sqrt{M_gp\max(p,n)} + \epsilon n}.
\end{equation}
Thus the condition \eqref{cond-Kb} is equivalent to 
\begin{equation}
K \geq \frac{4\kappa L_g\sqrt{M_gp\max(p,n)}} {m_gn\epsilon} \cdot {\log\left(Q''(p,\epsilon)\right)}
 = \frac{ \log\left(Q''(p,\epsilon)\right)}{m_gnh}.
\end{equation}
From  the analysis shown above, this yields that 
\begin{equation}
(1-nm_gh)^KW_2(\nu_{0,h,b},\pi) \leq 0.3\epsilon.
\end{equation}

Let us proceed to the second component, $\nicefrac{L_ghp}{2m_g}$. From the formula of $h$, which is given in the statement of the theorem, 
\begin{equation}
\frac{L_ghp}{2m_g} = \frac{L_gp}{2m_g}\cdot \frac{\epsilon}{4\kappa L_g\sqrt{M_gp\max(p,n)}} \leq \frac{\epsilon}{8}.
\end{equation}
The latter inequality is true, if we assume that $L_g$, $M_g$ and
 $\kappa$ are 
greater than 1. Similarly,
\begin{equation}
\frac{11M_g^{\frac32} h\sqrt{pn}}{5m_g} = \frac{11M_g^{\frac32}
\sqrt{pn}}{5m_g}\cdot \frac{\epsilon}{4\kappa L_g\sqrt{M_gp\max(p,n)}} \leq \frac{11\epsilon}{20}.
\end{equation}
Summing up these three inequalities we obtain that, $W_2(\nu_{K,h,b},\pi)\leq \epsilon$.
\end{proof}

\section{Conclusion}\label{sec:6}
In this paper we have introduced a new mathematical framework which helps 
to analyze Stochastic Gradient Descent as a sampling method, where the 
potential function is strongly convex and gradient-Lipschitz. 
Considering the particular case, where the stochastic term is a normal 
random vector with a diagonal covariance matrix, we have shown a convergence rate of 
$\tilde{O}(p/\epsilon^2)$. The latter is a massive improvement compared to the classic LMC which was giving only $\tilde{O}(np/\epsilon^2)$. In 
the case when we also assumed second-order smoothness, we have got 
$\tilde{O}(p/\epsilon^2 \bigwedge \kappa \sqrt{np}/\epsilon)$ convergence rate.

\section*{Appendix: The choice of the noise variance}
In this section we give a little insight on why and how we chose the distribution of the noise vectors in \Cref{sec:4}. 
Suppose we have a set of $n$ numbers $A = \{a_1,a_2,\ldots,a_n\}$. A random variable $X$ 
is designed in the following way. We take a uniformly random subset $I$ of $A$ with a fixed 
size $b$ from the class $C_b$ of all subsets of fixed size $b$.  
Afterwards we calculate the value of $\frac{n}{b}\sum_{i\in I} a_i$ and assign it to $X$. 
One can easily claim that $\E[X] = \sum_{i=1}^{n}a_i$ and therefore if we assume $a_i$'s to be of the same order, then $\E[X] = O(n)$. Important detail to notice is that it does not depend on $b$. Unfortunately the order of the variance is not that easy to guess, so we will hereby calculate it.

\begin{proposition}
Let us define the variance of $X$ by $\mathbb{V}[X]$. Then
\begin{equation}
\mathbb{V}[X] =  O\left(\frac{n(n-b)}{b}\right).
\end{equation}
\end{proposition}
\begin{proof}
\begin{align*}
\mathbb{V}[X] =& \frac{1}{C_n^b} \sum_{I\in C_b}\left[\sum_{i\in I} a_i\right]^2 - \left(\sum_{i=1}^{n}a_i\right)^2\\
=& \frac{n^2}{b^2C_n^b} \sum_{I\in C_b}\left[\sum_{i\in I} a_i^2 + \sum_{i \neq j; i,j\in I} 2a_ia_j\right] - \left(\sum_{i=1}^{n}a_i\right)^2\\
=& \frac{n^2C_{n-1}^{b-1}}{b^2C_n^b} \sum_{i=1}^{n}a_i^2 + \frac{n^2C_{n-2}^{b-2}}{b^2C_n^b}\sum_{i \neq j} a_ia_j - \sum_{i=1}^{n}a_i^2 - \sum_{i \neq j}a_ia_j\\
=& \frac{n-b}{b} \sum_{i=1}^{n}a_i^2 + \frac{b-n}{nb-b}\sum_{i \neq  j} a_ia_j. \\
\end{align*}
We know that $\sum_{i=1}^{n}a_i^2 = O(n)$ and $\sum_{i \neq j} 2a_ia_j = O\big(n(n-1)\big)$. Therefore the order of the variance is 
\begin{equation}
O\left(\frac{n(n-b)}{b}\right).
\end{equation}
\end{proof}

\printbibliography[heading=bibintoc]

\end{document}